\documentclass[reqno,10pt,a4paper]{amsart}
\usepackage{amstext}
\usepackage{amsthm}
\usepackage{amsopn}
\usepackage{amsfonts}
\usepackage{amsmath}
\usepackage{amssymb}
\usepackage{amsthm}
\usepackage{amscd}
\usepackage{enumerate}
\usepackage{color}
\usepackage{epsfig}
\usepackage{graphicx}

\def\nc{\newcommand}
\newcommand{\N}{\mathbb{N}}

\newcommand{\R}{\mathbb{R}}

\newcommand{\eps}{\varepsilon}

\newcommand{\calB}{\mathcal B}

\def\eps{\varepsilon}

\def\ohe{{\overline{h}}_\eps}

\def\uh{{\underline{h}}}
\def\barh{{\overline{h}}}
\def\underh{{\underline{h}}}
\def\EeL{{E_{\varepsilon,L}}}

\newcommand{\tnorm}[1]{%
  {\left\vert\kern-0.25ex\left\vert\kern-0.25ex\left\vert {#1} 
   \right\vert\kern-0.25ex\right\vert\kern-0.25ex\right\vert}}

\nc{\Halmos}    {\ \raisebox{0.6ex}  {\framebox[0.9ex]{
			   \rule[0ex]{0ex}{0.5ex}
			    }}}

\theoremstyle{plain}
\newtheorem{theorem}{Theorem}[section]
\newtheorem{proposition}[theorem]{Proposition}
\newtheorem{lemma}[theorem]{Lemma}
\newtheorem{corollary}[theorem]{Corollary}

  \newtheoremstyle{TheoremNum}
        {}             
        {}              
        {\itshape}                      
        {}                              
        {\bfseries}                     
        {.}                             
        { }                             
        {\thmname{#1}\thmnote{ \bfseries #3}}
   \theoremstyle{TheoremNum}

\theoremstyle{definition}
\newtheorem{definition}[theorem]{Definition}
\newtheorem{remark}[theorem]{Remark}

\newtheoremstyle{remarkbreak}
  {}
  {}
  {\upshape}
  {}
  {\bfseries}
  {.}
  {\newline}
  {}

\theoremstyle{remarkbreak}

\nc{\dual}[1]{\langle #1 \rangle}

\def\bfone{{\mathbf 1}}

\nc{\Tr}{\mbox{\rm Tr}}
\nc{\mytrace}[1]{\Tr \left( \sstrut #1 \right)}

\newcommand{\Lbb}{\mathbb{L}}
\newcommand{\dd}{\;{\rm d}}

\nc{\mat}[2]  {\left(  \! \begin{array}{#1} #2 \end{array}\! \right)}

\begin{document}

\title {Ratio ergodic theorems: From Hopf to Birkhoff and Kingman}

\author[H.H. Rugh]{Hans Henrik Rugh}
\address{Hans Henrik Rugh\\
Laboratoire de Math\'ematiques d'Orsay, 
Universit\'e Paris-Sud, CNRS, 
Universit\'e Paris-Saclay\\
91405 Orsay Cedex, France.
Email:  hans-henrik.rugh@math.u-psud.fr}

\author[Damien Thomine]{Damien Thomine}
\address{Damien Thomine\\
Laboratoire de Math\'ematiques d'Orsay, 
Universit\'e Paris-Sud, CNRS, 
Universit\'e Paris-Saclay\\
91405 Orsay Cedex, France.
Email:  damien.thomine@math.u-psud.fr}


\begin{abstract} 
Hopf's ratio ergodic theorem has an inherent symmetry
which we exploit to provide a simplification of standard
proofs of Hopf's and Birkhoff's ergodic theorems.
We also present a ratio ergodic theorem 
for conservative transformations 
on a $\sigma$-finite measure space, generalizing
Kingman's ergodic theorem for subadditive sequences
and generalizing previous results by Akcoglu and Sucheston. 
\end{abstract}

\maketitle

\section{Introduction and statement of results}

Birkhoff's pointwise ergodic theorem~\cite{Birkhoff:1931} 
is a key tool in ergodic theory. 
It admits many notable generalizations, 
including Hopf's ratio ergodic theorem~\cite{Hopf:1937}, 
Kingman's subadditive ergodic theorem~\cite{Kingman:1968}, 
and more recently Karlsson-Ledrappier and Gou\"ezel-Karlsson theorems 
on cocycles of isometries~\cite{KarlssonLedrappier:2006,GouezelKarlsson:2015}. 
Since the work of Kamae~\cite{Kamae:1982} and Katznelson and
Weiss~\cite{KatznelsonWeiss:1982}, 
there exist very short and easy proofs of Birkhoff's ergodic theorem.
These proofs have also been adapted to Hopf's and Kingman's ergodic theorems 
(\cite{KamaeKeane:1997} and~\cite{KatznelsonWeiss:1982,Steele:1989},
respectively).

\smallskip

In this article, we 
provide a proof of Hopf's and Kingman's ergodic theorem, 
in the context of conservative transformations preserving
a $\sigma$-finite measure.
We follow the argument of Katznelson and Weiss~\cite{KatznelsonWeiss:1982}
but add a noticeable twist:
the statement of the ratio ergodic theorem 
has a natural symmetry which is not
present in Birkhoff's ergodic theorem, a symmetry
which can be 
leveraged to simplify proofs. This makes, in our opinion, 
Hopf's theorem more fundamental, 
with Birkhoff's theorem now appearing as a corollary 
(the inverse point of view is given
in e.g.~\cite{Zweimuller:2004}, where Hopf's theorem 
is deduced from Birkhoff's by inducing).

\smallskip

As for the Kingman ratio ergodic theorem
on a $\sigma$-finite measure space, a similar result
was obtained by Akcoglu and Suchestom~\cite{AkcogluSucheston:1978}
 under an additional 
integrability assumption.
Our result does not make this
assumption and the proof is significantly simpler (in our opinion).
As the reader may note there are
no significant complications coming from working
with $\sigma$-finite measures, 
but some parts may be simplified quite a lot
if one assumes ergodicity.

\begin{definition}

 Consider a $\sigma$-finite measure space $(X,\calB,\mu)$ and a measure preserving
 transformation $T:X\to X$. The transformation $T$ is said to be \emph{conservative} if:
 \begin{equation}
  \forall A\in \calB \; : \ \ \mu(A)>0 \ \ \Rightarrow 
  \exists n\geq 1: \mu(A\cap T^{-1} A)>0 .
 \end{equation}
 A subset $A \subset X$ is \emph{$T$-invariant} if $T^{-1}A = A$, and a function 
 $f : X \to \R$ is \emph{$T$-invariant} if $f \circ T = f$. The transformation $T$ 
 is \emph{ergodic} (for $\mu$) if for any measurable $T$-invariant subset $A$, 
 either $\mu (A) = 0$ or $\mu (X \setminus A) = 0$.
\end{definition}

Unless stated otherwise, we make throughout the standard assumption that 
$(X,\calB,\mu,T)$ is a conservative measure-preserving transformation on
a $\sigma$-finite measure space. Most of the time, ergodicity shall not be assumed. 


\smallskip

Given any $f : X \to \R$, we write $S_n f := \sum_{k=0}^{n-1} f \circ T^k$ 
for the Birkhoff sums. Our first goal is to give a proof of the following well-known:

\begin{theorem}[Hopf's ratio ergodic theorem]
\label{thm:Hopf}

 Let $(X,\calB,\mu,T)$ be a conservative, measure preserving transformation 
 on a $\sigma$-finite measure space. Let $f$, $g\in \Lbb^1(X,\mu)$ with $g > 0$ almost everywhere. 
 Then the following limit exists $\mu$-almost everywhere:
 \begin{equation}
   h 
   := \lim_{n \to \infty} \frac{S_n f}{S_n g}.
 \end{equation}
 The function $h$ is finite $\mu$-almost everywhere, $T$-invariant, and for any $T$-invariant subset $A\in \calB$:
 \begin{equation}
   \int_A f \dd \mu 
   = \int_A h g \dd \mu.
 \end{equation}
\end{theorem}

Note that, when $f>0$ as well, there is a 
natural symmetry between $f$ and $g$, 
which we will exploit in our proof.
This symmetry is lost in Birkhoff's version, where $g\equiv 1$. 

\smallskip

We will proceed 
to prove
a ratio version of Kingman's theorem for subadditive sequences.
Recall that a sequence $(a_n)_{n\geq 1}$ of measurable functions
is said to be subadditive (with respect
to $T$) if for $n$, $m\geq 1$, we have $\mu$-almost everywhere:
\begin{equation}
 a_{n+m} 
 \leq a_n + a_m \circ T^n.
\end{equation}

\begin{theorem} [A Kingman ratio ergodic theorem]
 \label{thm:Kingman}
 
 Let $(X,\calB,\mu,T)$ be a conservative, measure preserving transformation 
 on a $\sigma$-finite measure space. 
  Let $(a_n)_{n\geq 1}$ be a sub-additive measurable sequence
 of functions with values in $[-\infty,+\infty]$,
 and with $(a_1)^+ \in \Lbb^1 (X, \mu)$. 
 Let $g\in \Lbb^1(X,\mu)$ with $g>0$ almost everywhere. 
 Then the following limit exists $\mu$-almost everywhere:
 \begin{equation*}
 h 
 := \lim_{n\to +\infty} \frac{a_n}{S_n g} \in [-\infty,+\infty).
 \end{equation*}
 If $A\in \calB$ is a $T$-invariant set, then
 \begin{equation*}
 \inf_n \frac{1}{n} \int_A a_n \dd \mu
 = \lim_n \frac{1}{n} \int_A a_n \dd \mu
 = \int_A h g \dd \mu 
 \in [-\infty,+\infty).
 \end{equation*}
\end{theorem}

For a similar version of this theorem, but under an 
additional uniform integrability condition on $a_n$, 
cf.~\cite{AkcogluSucheston:1978}. 
To our knowledge the theorem is new in the stated generality.

\smallskip

The remainder of this article is organized as follows. In Section~\ref{sec:MainLemmas} 
we prove some classical lemmas about conservative dynamical systems, and the main 
lemma (Lemma~\ref{lem:Principal}, which slightly generalizes the main theorem of~\cite{KatznelsonWeiss:1982}). 
In Section~\ref{sec:Hopf} we prove Hopf's ratio ergodic theorem, and in Section~\ref{sec:Kingman} 
the above Kingman ratio ergodic theorem
(whose proof uses Hopf's ergodic theorem).

\section{Main lemmas}
\label{sec:MainLemmas}

First, and so that our proofs will be essentially self-contained, 
let us state and prove some consequences of conservativity. 
For details the reader may consult e.g. \cite[Chap 1]{Aaronson:1997}.
Conservativity is an a priori mild recurrence condition which
is equivalent to 
the following seemingly stronger recurrence condition. 
Let $(X, \calB, \mu, T)$ be a conservative, measure-preserving 
dynamical system. Then, given any measurable subset $A$, 
almost everywhere on $A$,
\begin{equation}
 \sum_{n\geq 0} \bfone_A \circ T^n
 = +\infty.
\end{equation}
In other words, almost every point in $A$ returns infinitely often to $A$.
A consequence is the following:
\begin{lemma}
\label{lem:Conservatif}

 Let $(X,\calB,\mu,T)$ be a measure-preserving and conservative transformation. Let $f:X \to \R_+$ be measurable. 
 Then, $\mu$-almost everywhere on $\{f >0\}$:
 \begin{equation}
  \label{eq:conserv}
  \lim_{n \to +\infty} S_n f 
  = +\infty.
 \end{equation}
\end{lemma}

\begin{proof}

For $n \geq 1$, let $A_n := \{f \geq \frac{1}{n}\}$. 
Let $\Omega_n := A_n \cap \bigcap_{m \geq 0} \bigcup_{k \leq m} T^{-k} A_n$ 
be the set of points of $A_n$ which return to $A_n$ infinitely many times.

\smallskip

Since $f$ is positive and takes value at least $1/n$ on $A_n$, Equation~\eqref{eq:conserv} 
follows for all $x \in \Omega_n$. Let $\Omega_\infty := \bigcup_{n\geq 1} \Omega_n$. Then 
Equation~\eqref{eq:conserv} holds on $\Omega_\infty$.
Since $(X,\calB,\mu,T)$ is assumed to be conservative, 
$\mu (\Omega_n \Delta A_n) = 0$,
so that $\mu (\Omega_\infty \Delta \bigcup_{n \geq 1} A_n) = 0$. 
But, $\bigcup_{n \geq 1} A_n = \{f >0\}$, so $\Omega_\infty$ has full measure in $\{f >0\}$.
\end{proof}

Let $(a_n)_{n \geq 1}$ be a super-additive sequence of functions and 
$g \in \Lbb^1 (A, \mu)$. We define for every $x\in X$ the following 
lower and upper limits:
 \begin{equation}
  0 \ \ \leq  \ \
  \underh(x) = \liminf_{n\to \infty} \frac{a_n(x)}{S_n g(x)} 
  \ \ \leq \ \
  \barh(x) = \limsup_{n\to \infty} \frac{a_n(x)}{S_n g(x)} 
  \ \ \leq  \ \ +\infty.
 \end{equation}
Both $\underh$ and $\barh$ are measurable and, in fact, a.e.
 $T$-invariant:

\begin{lemma}
 \label{lem:InvarianceLimits}
 
 Let $(X,\calB,\mu,T)$ be a measure-preserving and conservative transformation. 
 Let $(a_n)_{n \geq 1}$ be a super-additive sequence of functions, with 
 $a_1 \geq 0$ a.e.. Let $g \in \Lbb^1 (A, \mu; \R_+^*)$.
 Then $\underh \circ T = \underh$ and $\barh \circ T = \barh$ almost everywhere.
\end{lemma}

\begin{proof}

 We prove the result for $\underh$;
 the proof for $\barh$ is essentially the same. 
 Let $(a_n)$ and $g$ be as in the lemma. Then:
 \begin{equation*}
  \frac{a_{n+1}}{S_{n+1} g} 
  \geq \frac{a_1 + a_n \circ T}{g+(S_n g)\circ T} 
 \end{equation*}
 By Lemma~\ref{lem:Conservatif}, $\lim_{n \to + \infty} S_n g = +\infty$ almost everywhere,
 whence, taking the liminf,
 \begin{equation*}
  \underh 
  \geq \underh \circ T.
 \end{equation*}

The function $\eta=\uh/(1+\uh)$ 
   takes values in $[0,1]$ and we have
 $A := \{\underh > \underh \circ T\} = \{\eta > \eta\circ T\}$. 
 By Lemma~\ref{lem:Conservatif},
 the map $\phi := \eta - \eta \circ T$ verifies
 $\lim_{n \to + \infty} S_n \phi = +\infty$
 almost everywhere on $A$.
 But as $S_n \phi = \eta - \eta \circ T^{n+1}\in [-1,1]$ everywhere,
 we must have $\mu (A)=0$.

\end{proof}


We can now state and prove our main lemma.

\begin{lemma}
 \label{lem:Principal}
 
 Let $(X,\calB,\mu,T)$ be a measure-preserving and conservative transformation. 
 Let $(a_n)_{n \geq 1}$ be a super-additive sequence of functions, with 
 $a_1 \geq 0$ almost everywhere. Let $g \in \Lbb^1 (A, \mu; \R_+^*)$.
 
 \smallskip
 
 Then $\barh$ is $T$-invariant, and for all $T$-invariant $A \in \calB$,
 \begin{equation}
  \label{eq:inequality}
  \liminf_{n \to + \infty} \frac{1}{n} \int_A a_n \dd \mu 
  \geq \int_A \barh g \dd \mu.
 \end{equation}
\end{lemma}

\begin{proof}

 By Lemma~\ref{lem:InvarianceLimits}, $\barh$ is $T$-invariant, but
 with values a priori in $[0,+\infty]$. For $\eps>0$, we set 
 \begin{equation*}
  \ohe 
  := \frac{\barh}{1 + \varepsilon \barh},
 \end{equation*}
 with the convention that $\ohe (x)=1/\eps$ when $\barh(x)=+\infty$. Let:
 \begin{equation*}
  n_\eps
  := \inf \left\{ \; k\geq 1: \; a_k \;\geq \;  S_k (\ohe g) \; \right\}.
 \end{equation*}
 Then $n_\eps(x)=1$ whenever $\barh(x)=0$.
 When $\barh(x)>0$, we have $\ohe (x)<\barh(x)$, so $n_\eps (x)$ is finite by the definition of the limsup.
 Introduce a time-cutoff $L\geq 1$ and denote 
 $E = \EeL := \{n_\eps \leq L \}$. We then set: 
 \begin{equation*}
  \varphi(x) 
  = \varphi_{\eps,L}(x) 
  := \left\{ \begin{array} {cc}
    1 & \text{ if }  x\notin E \\
    n_\eps(x) & \text{ if } x \in E
  \end{array} \right. .
 \end{equation*}
 
 One verifies that for every $x\in X$:
 \begin{equation*}
  a_{\varphi(x)} 
  \geq S_{\varphi(x)} (g \ohe \bfone_{E}).
 \end{equation*}
 When $x\in E$ this is true by the very definition of $n_\eps(x)$, 
 while for $x\notin E$ it holds because the right hand side vanishes.
 
 \smallskip
 
 Define a sequence of stopping times:
 \begin{equation*}
  \left\{ \begin{array}{ll}
    \tau_0(x) & =0, \\
    \tau_{k+1}(x)& = \tau_k(x) + \varphi \left( T^{\tau_k(x)} x\right), \quad k\geq 0. 
  \end{array} \right. .
 \end{equation*}
 Note that  $1 \leq \tau_{k+1}-\tau_k \leq L$ for every $k\geq 0$, and for all $x \in X$:
 \begin{equation*}
  a_{\tau_k(x)} (x) 
   \geq \sum_{j=0}^{k-1} a_{\varphi (x)} \circ T^{\tau_j}(x) 
   \geq \sum_{j=0}^{k-1} S_{\varphi (x)} (\ohe g\bfone_E) \circ T^{\tau_j}(x)
   = S_{\tau_k(x)} (g \ohe  \bfone_E) (x).
 \end{equation*}
 
 Let $N \geq 1$ and $x\in X$. There exists $k\geq 1$ such that
 $N < \tau_k(x) \leq N+L$. Then:
 \begin{equation*}
  a_{N+L} (x) 
  \geq a_{\tau_k(x)} (x) + \sum_{i = \tau_k (x)+1}^{N+L} a_1 \circ T^i
  \geq S_{\tau_k(x)} (\ohe g \bfone_E) (x)
  \geq S_N (g\ohe \bfone_E) (x),
 \end{equation*}
 which allows us to get rid of the intermediate stopping times.
 Take now a $T$-invariant set $A\in \calB$ and integrate the 
 above inequality over $A$. By $T$-invariance:
 \begin{equation*}
  \frac{1}{N+L} \int_A a_{N+L} \dd \mu 
  \geq \frac{1}{N+L} \int_A S_N(\ohe g \bfone_E) \dd \mu
  = \frac{N}{N+L} \int_A \ohe g \bfone_E \dd \mu.
 \end{equation*}
 Letting $N \to +\infty$, we conclude that:
 \begin{equation*}
  \liminf_{n \to + \infty} \frac{1}{n} \int_A a_n \dd \mu 
  \geq \int_A \ohe g \bfone_E \dd \mu.
 \end{equation*}
 Letting $L \to+\infty$ and finally $\eps\to 0$, we obtain by monotone convergence:
 \begin{equation*}
  \liminf_{n \to + \infty} \frac{1}{n} \int_A a_n \dd \mu 
  \geq \int_A \barh g \dd \mu.
 \end{equation*}
\end{proof}

\section{Proof of Hopf's Ratio ergodic theorem}
\label{sec:Hopf}

We are now ready to prove to prove Hopf's ergodic theorem. Lemma~\ref{lem:Principal} 
only provides a upper bound on $\barh$. In most proofs using these techniques, 
the lower bound on $\underh$ follows by repeating the same argument, with a modified 
stopping time (and some handwaving). 
Here we notice that the symmetry of Hopf's ratio 
ergodic theorem provides us with a shortcut:

\begin{proof}[Proof of Theorem~\ref{thm:Hopf}]

Let $f$, $g \in \Lbb^1 (X, \mu; \R_+^*)$. By Lemma~\ref{lem:Principal}, with $a_n = S_n f$, 
for any $T$-invariant measurable $A$:
\begin{equation*}
 \int_A f \dd \mu 
 \geq \int_A \barh g \dd \mu.
\end{equation*}
In particular, taking $A = X$, we see that $\barh g \in \Lbb^1 (X, \mu)$. 

\smallskip

We now use the symmetry, and apply Lemma~\ref{lem:Principal} with $g$ and $f$.
Since $$\limsup_{n \to + \infty} (S_n g)/(S_n f) = \underh^{-1},$$
we get $\underh^{-1} f \in \Lbb^1 (X, \mu)$, and in particular $+ \infty > \barh \geq \underh >0$ almost everywhere.

\smallskip

We now apply Lemma~\ref{lem:Principal} again, with $\barh g$ and $f$.
Since 
\begin{equation*}
 \limsup_{n \to + \infty} \frac{S_n (\barh g)}{S_n f} 
 = \frac{\barh}{\underh},
\end{equation*}
we get that, for any $T$-invariant measurable $A$,
\begin{equation*}
 \int_A f \dd \mu 
 \geq \int_A \barh g \dd \mu 
 \geq \int_A f \frac{\barh}{\underh} \dd \mu.
\end{equation*}
As $f>0$ a.e.\ and the
integral is finite we conclude that $\barh = \underh =: h$ almost everywhere.

\smallskip

Let us now turn towards the proof of Theorem~\ref{thm:Hopf}
without positivity assumption on $f$. 
Since $\mu$ is $\sigma$-finite, we can write $f= f_+ - f_-$, 
with $f_+$ and $f_-$ in $\Lbb^1 (X, \mu; \R_+^*)$. Then, $\mu$-almost everywhere:
\begin{equation*}
 \lim_{n \to +\infty} \frac{S_n f}{S_n g} 
 = \lim_{n \to \infty} \frac{S_n f_+}{S_n g} -\frac{S_n f_-}{S_n g}
 = h_+ - h_- 
 =: h .
\end{equation*}
In addition, $\int_A f \dd \mu = \int_A (f_+-f_-) \dd \mu = \int_A (h_+-h_-)g \dd \mu = \int_A h g \dd \mu$.
\end{proof}

\begin{remark}
 
 The proof of Theorem~\ref{thm:Hopf} itself can be significantly shortened if one assumes 
 that $(X,\calB,\mu,T)$ is ergodic: since $\barh$ and $\underh$ are then constant, applying 
 Lemma~\ref{lem:Principal} to the pairs $(f,g)$ and $(g,f)$ yields directly:
 \begin{equation*}
  \underh 
  \geq \frac{\int_X f \dd \mu}{\int_X g \dd \mu} 
  \geq \barh.
 \end{equation*}
 
 \smallskip
 
 There is, to our knowledge, not much gain to be had in the proof of Lemma~\ref{lem:Principal}. 
\end{remark}

In the ergodic case, the statement of Hopf's ergodic theorem can be simplified.

\begin{corollary}[Hopf's theorem, ergodic version]
\label{cor:HopfErgodic}

 Let $(X,\calB,\mu,T)$ be a measure-preserving, conservative and ergodic transformation. 
 Let $f$, $g \in \Lbb^1 (A, \mu)$ with $\int_X g \dd \mu \neq 0$. 
 Then $\mu$-almost everywhere:
 \begin{equation}
  \lim_{n\to +\infty} \frac{S_n f}{S_n g}
  = \frac{\int_X f \dd \mu}{\int_X g \dd \mu}.
 \end{equation}
\end{corollary}

\begin{proof}

We decompose as above $f=f_+-f_-$, with $f_\pm$ integrable and positive.
By the Theorem~\ref{thm:Hopf}, $\mu$-almost everywhere,
\begin{equation*}
 \lim_{n \to + \infty} \frac{S_n g}{S_n f_\pm} 
 = k_\pm.
\end{equation*}
By ergodicity, the $k_\pm$ are constant and then non-zero,
since $\int_X g \dd \mu = k_\pm \int_X f_\pm \dd \mu \neq 0$. Thus, almost everywhere:
\begin{equation*}
 \lim_{n\to \infty} \frac{S_n f}{S_n g} 
 = \lim_{n\to \infty} \frac{S_n f_+}{S_n g}-\frac{S_n f_-}{S_n g}
 = \frac{1}{k_+} - \frac{1}{k_-} 
 = \frac{\int_X (f_+ - f_-) \dd \mu}{\int_X g \dd \mu} 
 = \frac{\int_X f \dd \mu}{\int_X g \dd \mu} . \qedhere
\end{equation*}
\end{proof}

As a special case, we may also consider when
$\mu$ is a probability measure and $g\equiv 1$ (thus integrable). 
$T$ is automatically conservative by Poincar\'e recurrence theorem. 
From Theorem~\ref{thm:Hopf} we deduce:

\begin{corollary}[Birkhoff's Ergodic Theorem]
\label{cor:Birkhoff}

 Let $(X,\calB,\mu,T)$ be a measure preserving
 transformation on a probability space. Let $f\in \Lbb^1(X,\mu)$.
 Then the following limit exists $\mu$-almost everywhere:
 \begin{equation}
  f^*
  := \lim_{n\to +\infty} \frac{1}{n} S_n f.
 \end{equation}
 $f^*$ is $T$-invariant (up to a set of measure $0$), and for any $T$-invariant measurable subset $A$:
 \begin{equation}
  \int_A f \dd \mu 
  = \int_A f^* \dd \mu.
 \end{equation}
\end{corollary}

\section{Kingman, $\sigma$-finite version}
\label{sec:Kingman}

We proceed here with a ratio version of Kingman's theorem for
non-negative super-additive sequences, from which Theorem~\ref{thm:Kingman} 
shall follow easily.

\begin{proposition}
 \label{prop:SuperAdditif}
 
 Let $(X,\calB,\mu,T)$ be a measure-preserving and conservative transformation. 
 Let $(a_n)_{n \geq 1}$ be a super-additive sequence of functions, with 
 $a_1 \geq 0$ almost everywhere. Let $g \in \Lbb^1 (A, \mu; \R_+^*)$.
 
 \smallskip
 
 Then the following limit exists $\mu$-almost everywhere:
 \begin{equation*}
  h 
  := \lim_{n \to + \infty} \frac{a_n}{S_n g}
  \in [0, + \infty].
 \end{equation*}
 In addition, for any $T$-invariant measurable set $A$,
 \begin{equation*}
  \sup_{n \in \N} \frac{1}{n} \int_A a_n \dd \mu 
  = \lim_{n \to + \infty} \frac{1}{n} \int_A a_n \dd \mu
  = \int_A h g \dd \mu 
  \in [0,+\infty].
 \end{equation*}
\end{proposition}

\begin{proof}

Let $A$ be any $T$-invariant measurable set. 
By Lemma~\ref{lem:Principal}, we know that:
\begin{equation*}
 \liminf_{n \to + \infty} \frac{1}{n} \int_A a_n \dd \mu
 \geq \int_A \barh g \dd \mu.
\end{equation*}
We want to prove the converse inequality (inverting the direction of the inequality, 
and the $\liminf$ and $\limsup$). Let $K < \sup_{n \geq 1} \frac{1}{n} \int_A a_n \dd \mu$.
Then we can find $k \geq 1$ such that $\frac{1}{k} \int_A a_k \dd \mu > K$. For $M >0$, 
let $f_{k, M} := \min\{a_k, MS_k g, Mk\}/k$. By the monotone convergence theorem, 
there exists $M > 0$ such that:
\begin{equation*}
 \int_A f_{k, M} \dd \mu 
 > K.
\end{equation*}
Let $n \geq 2k$, and let $q$, $r$ be such that $n = qk+r$ and $0 \leq r < k$. Then:
\begin{align*}
a_n 
& = \sum_{i=0}^{k-1} \frac{a_n}{k} 
\geq \sum_{i=0}^{k-1} \frac{1}{k} \left(a_i + \sum_{j=0}^{q-2} a_k \circ T^{i+jk} + a_{n-i-(q-1)k} \circ T^{i+(q-1)k} \right) \\
& \geq \sum_{i=0}^{(q-1)k-1} \frac{a_k}{k} \circ T^i 
= S_{(q-1)k} (a_k / k) 
\geq S_{n-2k} f_{k, M}.
\end{align*}
Let $\underh_{k,M} := \liminf_{n \to + \infty} S_n f_{k, M} / S_n g$. Note that 
$|S_n f_{k,M} - S_{n-2k} f_{k, M}| \leq 2kM$, so that $\underh_{k,M} = \liminf_{n \to + \infty} S_{n-2k} f_{k, M} / S_n g$ 
by Lemma~\ref{lem:Conservatif}. Hence, $\underh \geq \underh_{k,M}$. By Hopf's theorem (cf.\ Theorem~\ref{thm:Hopf}), 
\begin{equation*}
 K 
 \leq \int_A f_{k,M} \dd \mu
 = \int_A \underh_{k,M} g \dd \mu 
 \leq \int_A \underh g \dd \mu.
\end{equation*}
Since this is true for all $K < \sup_{n \geq 1} \frac{1}{n} \int_A a_n \dd \mu$, we finally get:
\begin{equation*}
 \limsup_{n \to + \infty} \frac{1}{n} \int_A a_n \dd \mu 
 \leq \sup_{n \geq 1} \frac{1}{n} \int_A a_n \dd \mu 
 \leq \int_A \underh g \dd \mu,
\end{equation*}
whence the sequence $(\frac{1}{n} \int_A a_n \dd \mu)_{n \geq 1}$ converges to its supremum, 
and:
\begin{equation}
 \label{eq:EgaliteIntegrales}
 \int_A \barh g \dd \mu 
 = \int_A \underh g \dd \mu.
\end{equation}
All is left is to prove that $\barh=\underh$ almost everywhere. For $M \geq 0$, take 
$A := \{\underh \leq M\}$. Then $\underh g$ is integrable on $A$,
and since $g>0$ a.e. Equation~\eqref{eq:EgaliteIntegrales} 
implies $\underh = \barh$ almost everywhere on $A$. Since this is true for all $M \geq 0$, 
we get that $\underh = \barh$ almost everywhere on $\{\underh < +\infty\}$, and 
obviously $\underh = \barh$ on $\{\underh = +\infty\}$.
\end{proof}

Let us finish the proof of Theorem~\ref{thm:Kingman}.

\begin{proof}[Proof of Theorem~\ref{thm:Kingman}]

Up to taking the opposite sequences, we work with super-additive sequences. 
Let $(a_n)_{n \geq 1}$ be a super-additive sequence, and $g$ a positive and integrable function. 
Write $a_1 = a_1^+ - a_1^-$ and $b_n := a_n + S_n a_1^-$. Then 
$(b_n)_{n \geq 1}$ and $g$ satisfy the hypotheses of Proposition~\ref{prop:SuperAdditif}, 
and so do $(S_n a_1^-)_{n \geq 1}$ and $g$. The (almost everywhere) limits and integrals 
concerning $(S_n a_1^-)_{n \geq 1}$ and $g$ are finite, so we can subtract them from the limits 
and integrals concerning $(b_n)_{n \geq 1}$ and $g$.
\end{proof}

\end{document}